\documentclass[12pt]{amsart}

\usepackage{times}
\usepackage{amsmath,amsfonts,amssymb,amsopn,amscd,amsthm}
\usepackage{mathbbol}
\usepackage{mathrsfs}
\usepackage{graphicx,xspace}
\usepackage{epsfig}
\usepackage{enumerate} 
\usepackage{enumitem}
\usepackage{xcolor}
\usepackage{caption} 
\usepackage{capt-of}
\usepackage[all]{xypic}

\usepackage[T1]{fontenc}
\usepackage[utf8]{inputenc}

\usepackage{pgf,tikz,pgfplots}
\usepackage{mathrsfs}
\usetikzlibrary{arrows}

\usepackage{hyperref}
\usepackage{psfrag}
\usepackage{bm}
\usepackage{youngtab}
\usepackage{tikz}
\usetikzlibrary{snakes}
\usepackage[all]{xy}
\usepackage{varioref}
\theoremstyle{plain}
\newtheorem{theoreme}{Theorem}[section]
\newtheorem{prop}[theoreme]{Proposition}

\newtheorem{lem}[theoreme]{Lemma}

\theoremstyle{remark}
\newtheorem{definition}[theoreme]{Definition}
\newtheorem{ex}[theoreme]{Example}
\newtheorem{remark}[theoreme]{Remark}
\newtheorem{c-ex}[theoreme]{Counter-example}

\date{}

\DeclareUnicodeCharacter{03C7}{\chi}
\DeclareUnicodeCharacter{B2}{^{2}}
\DeclareMathOperator{\supp}{supp}

\DeclareMathOperator{\Proj}{Proj}
\DeclareMathOperator{\Fil}{Fil}

\DeclareMathOperator{\ty}{type}

\DeclareMathOperator{\Id}{Id}
\DeclareMathOperator{\GL}{GL}

\author[O. Tout]{Omar Tout}
\address{Department of Mathematics, College of Science, Sultan Qaboos University, P. O Box 36, Al Khod 123, Sultanate of Oman}
\email{o.tout@squ.edu.om}

\title[]{The algebra of conjugacy classes of the wreath product\\ of a finite group with the symmetric group}


\makeatletter
\@namedef{subjclassname@2020}{%
  \textup{2020} Mathematics Subject Classification}
\makeatother


\begin{document}
\maketitle

\begin{abstract} For a finite group $G,$ we define the concept of $G$-partial permutation and use it to show that the structure coefficients of the center of the wreath product $G\wr \mathcal{S}_n$ algebra are polynomials in $n$ with non-negative integer coefficients. Our main tool is a combinatorial algebra which projects onto the center of the group $G\wr \mathcal{S}_n$ algebra for every $n.$ This generalizes the Ivanov and Kerov method to prove the polynomiality property for the structure coefficients of the center of the symmetric group algebra. 
\end{abstract}

\section{Introduction}
Throughout this paper $G$ will be a finite group, $1_G$ its identity element, $G_\star$ its set of conjugacy classes and $G^\star$ its set of irreducible complex characters. If $n$ is a positive integer, let $\mathcal{S}_n$ denote the symmetric group on the set $[n] := \lbrace 1,2,\ldots ,n\rbrace.$ A partition is a finite list of non-increasing positive integers called parts. The size of a partition is the sum of all of its parts. A partition of size $n$ is usually called a partition of $n.$ We denote by $\mathcal{P}_n^{G_\star}$ the set of families of partitions $\Lambda=(\Lambda(c))_{c\in G_\star},$ indexed by $G_\star,$ such that the sizes of the partitions $\Lambda(c)$ sum up to $n.$ The type of an element of the wreath product $G\wr \mathcal{S}_n$ is a family of partitions $\Lambda\in \mathcal{P}_n^{G_\star},$ see \cite{McDo}. Two elements of $G\wr \mathcal{S}_n$ are conjugate if and only if they have the same type. The center of the group $G\wr \mathcal{S}_n$ algebra, which will be denoted $Z(\mathbb{C}[G\wr \mathcal{S}_n]),$ is the algebra over $\mathbb{C}$ generated by the conjugacy classes of $G\wr \mathcal{S}_n.$ If $\Lambda\in \mathcal{P}_n^{G_\star},$ we define $\mathbf{C}_\Lambda$ to be the formal sum of all the elements in $G\wr \mathcal{S}_n$ with type $\Lambda.$ The family $(\mathbf{C}_\Lambda)_\Lambda$ indexed by $\mathcal{P}_n^{G_\star}$ is a linear basis for $Z(\mathbb{C}[G\wr \mathcal{S}_n]).$ The structure coefficients $c_{\Lambda\Delta}^{\Gamma}$ are the non-negative integers defined by the following product in $Z(\mathbb{C}[G\wr \mathcal{S}_n])$
\begin{equation*}
\mathbf{C}_\Lambda \mathbf{C}_\Delta=\sum_{\Gamma \in \mathcal{P}_n^{G_\star}} c_{\Lambda\Delta}^{\Gamma} \mathbf{C}_\Gamma.
\end{equation*}
In the case where $G$ is the trivial group, the group $G\wr \mathcal{S}_n$ is isomorphic to the symmetric group $\mathcal{S}_n.$ The conjugacy classes of $\mathcal{S}_n$ are indexed by partitions of $n.$ It is a difficult problem to find explicit formulas even for particular structure coefficients of $Z(\mathbb{C}[\mathcal{S}_n]),$ see \cite{katriel1987explicit}, \cite{GoupilSchaefferStructureCoef}, \cite{Toutjacodesmath}. In \cite{FaharatHigman1959}, Farahat and Higman showed that the structure coefficients of $Z(\mathbb{C}[\mathcal{S}_n])$ are polynomials in $n.$ By introducing partial permutations in \cite{Ivanov1999}, Ivanov and Kerov gave a combinatorial proof to this result. Recently, we used in \cite{Tout2018} our general framework developed in \cite{Tout2017} to show a polynomiality property for the structure coefficients of $Z(\mathbb{C}[\mathcal{S}_k\wr \mathcal{S}_n]).$ \\

Beside from being combinatorial, the Ivanov-Kerov approach, developed in \cite{Ivanov1999}, uses a universal algebra which turns out to be isomorphic to the algebra of shifted symmetric functions. In the past few years, it was used to show a polynomiality property for the structure coefficients of some interesting algebras. For example, we define the notion of partial bijection in \cite{toutejc} to show that the structure coefficients of the Hecke algebra of the pair $(\mathcal{S}_{2n},\mathcal{B}_n),$ where $\mathcal{B}_n$ is the hyperoctahedral subgroup of $\mathcal{S}_{2n},$ are polynomials in $n.$ In \cite{meliot2013partial}, the concept of partial isomorphism appeared to give a polynomiality property for the structure coefficients of the center of the group $\GL(n,\mathbb{F}_q)$ algebra, where $q$ is a prime number and $\GL(n,\mathbb{F}_q)$ is the group of invertible $n\times n$ matrices with coefficients in the finite field $\mathbb{F}_q.$ We used the notion of $k$-partial permutation in \cite{tout2020k} to give a more combinatorial proof to our result in \cite{Tout2018}.\\

In \cite{wang2004farahat}, Wang proved that the structure coefficients $c_{\Lambda\Delta}^{\Gamma}$ of $Z(\mathbb{C}[G\wr \mathcal{S}_n])$ are polynomials in $n.$ He used the Farahat-Higman approach developed in \cite{FaharatHigman1959} for the center of the symmetric group algebra. The goal of this paper is to generalize the Ivanov-Kerov approach in order to obtain Wang's result by a more algebraic combinatorial way. For this reason, we will define the concept of $G$-partial permutation and use it to build a universal combinatorial algebra which projects onto the center of the group $G\wr \mathcal{S}_n$ algebra for each $n.$ We will prove that this universal algebra is isomorphic to the algebra of shifted symmetric functions on $|G^\star|$ alphabets. Recently, it came to our attention that Wang mentioned our generalization in \cite[Section $5.3$]{wang2004vertex}. However, in addition to providing all the details, we think that some presented results like Theorem \ref{th:7.1}, are new and make a valuable contribution to the literature. \\

The paper is organized as follows. In Section \ref{sec_2}, we present the necessary definitions for partitions and we review some basic results concerning the conjugacy classes and the center of the group $G\wr \mathcal{S}_n$ algebra. Then, in Section \ref{sec_3}, we introduce the notion of $G$-partial permutation. An action of the group $G\wr \mathcal{S}_n$ on the set of $G$-partial permutations of $n$ is given in Section \ref{sec_4}. The universal combinatorial algebra $\mathcal{A}_\infty^G,$ which projects on the center of the group $G\wr \mathcal{S}_n$ algebra for each $n,$ will be built in Section \ref{sec_5}. Next in Section \ref{sec_6}, we prove in Theorem \ref{main_theorem} that the structure coefficients of the center of the group $G\wr \mathcal{S}_n$ algebra are polynomials in $n.$ In the last section, we present an isomorphism between $\mathcal{A}_\infty^G$ and the algebra of shifted symmetric functions on $|G^\star|$ alphabets.

\section{Algebra of the conjugacy classes of $G \wr \mathcal{S}_n$}\label{sec_2}
In this section we will review all necessary definitions and results concerning the conjugacy classes of $G \wr \mathcal{S}_n.$ For more details, the reader is invited to check \cite[Appendix B]{McDo}.

\subsection{Partitions} A \textit{partition} $\lambda$ is a weakly decreasing list of positive integers $(\lambda_1,\ldots,\lambda_l).$ The $\lambda_i$ are called the \textit{parts} of $\lambda.$ The \textit{size} of $\lambda,$ denoted by $|\lambda|,$ is the sum of all of its parts. We say that $\lambda$ is a partition of $n$ if $|\lambda|=n.$ The set of all partitions of $n$ will be denoted $\mathcal{P}_n.$ In this paper, we will mainly use the exponential notation $\lambda=(1^{m_1(\lambda)},2^{m_2(\lambda)},3^{m_3(\lambda)},\ldots),$ where $m_i(\lambda)$ is the number of parts equal to $i$ in the partition $\lambda.$ We will dismiss $i^{m_i(\lambda)}$ from $\lambda$ when $m_i(\lambda)=0,$ for example, we will write $\lambda=(1,2,4^2)$ instead of $\lambda=(1,2,3^0,4^2,5^0,\ldots).$ If $\lambda$ and $\delta$ are two partitions, we define the \textit{union} $\lambda \cup \delta$ and subtraction $\lambda \setminus \delta$ (if exists) as the following partitions:
$$\lambda \cup \delta=(1^{m_1(\lambda)+m_1(\delta)},2^{m_2(\lambda)+m_2(\delta)},3^{m_3(\lambda)+m_3(\delta)},\ldots).$$
$$\lambda \setminus \delta=(1^{m_1(\lambda)-m_1(\delta)},2^{m_2(\lambda)-m_2(\delta)},3^{m_3(\lambda)-m_3(\delta)},\ldots) \text{ if $m_i(\lambda)\geq m_i(\delta)$ for any $i.$ }$$

The \textit{cycle-type} of a permutation of $\mathcal{S}_n$ is the partition of $n$ obtained from the lengthes of the cycles that appear in its decomposition into product of disjoint cycles. For example, the permutation $(1,4)(2,6,3)(5)(7,8)$ of $\mathcal{S}_{8}$ has cycle-type $(1,2^2,3).$ It is well known that two permutations of $\mathcal{S}_n$ belong to the same conjugacy class if and only if they have the same cycle-type. Thus the conjugacy classes of $\mathcal{S}_n$ can be indexed by partitions of $n.$ The conjugacy class of $\mathcal{S}_n$ associated to the partition $\lambda=(1^{m_1(\lambda)},2^{m_2(\lambda)},3^{m_3(\lambda)},\ldots,n^{m_n(\lambda)})\in \mathcal{P}_n$ will be denoted $C_\lambda$ and its cardinal is given by the following formula:
$$|C_\lambda|=\frac{n!}{z_\lambda},$$
where
\begin{equation*}
z_\lambda:=\prod_{i\geq 1} i^{m_i(\lambda)}m_i(\lambda)!.
\end{equation*}

A partition is called \textit{proper} if it does not have any part equal to 1. The proper partition associated to a partition $\lambda$ is the partition $\bar{\lambda}:=(2^{m_2(\lambda)},3^{m_3(\lambda)},\ldots).$ The set of all proper partitions with size less than or equal to $n$ will be denoted $\mathcal{PP}_{\leq n}.$ If $\lambda$ is a partition of $r<n,$ we can extend $\lambda$ to a partition of $n$ by adding $n-r$ parts equal to one, the new partition of $n$ will be denoted $\underline{\lambda}_n.$

If $X$ is a finite set and $\Lambda=(\Lambda(x))_{x\in X}$ is a family of partitions indexed by $X,$ we define the size of $\Lambda,$ denoted by $|\Lambda|,$ to be the sum of the sizes of $\Lambda(x)$
\begin{equation*}
|\Lambda|=\sum_{x\in X}|\Lambda(x)|.
\end{equation*}
The set of families of partitions of size $n$ indexed by $X$ will be denoted $\mathcal{P}_n^X$ and $\mathcal{P}_{\leq n}^X$ will denote the set of families of partitions indexed by $X$ with size less than or equal to $n.$ In this paper we will mainly encounter families of partitions indexed by $G_\star$ and $G^\star.$  An element $\Lambda\in \mathcal{P}_n^{G_\star}$ is called proper if the partition $\Lambda (\lbrace 1_G\rbrace)$ is proper. We will use $\mathcal{PP}_n^{G_\star}$ (resp. $\mathcal{PP}_{\leq n}^{G_\star}$) to denote the set of proper families of partitions of size $n$ (resp. less than or equal to $n$) indexed by $G_\star.$ If $\Lambda \in\mathcal{P}_{\leq n}^{G_\star},$ we define $\underline{\Lambda}_n$ to be the element of $\mathcal{P}_n^{G_\star}$ with $\underline{\Lambda}_n (c) =\Lambda (c)$ if $c\neq \lbrace 1_G\rbrace$ and
$$\underline{\Lambda}_n (\lbrace 1_G\rbrace) =\Lambda (\lbrace 1_G\rbrace)\cup (1^{n-|\Lambda|})=\underline{\Lambda (\lbrace 1_G\rbrace)}_n.$$

\subsection{Conjugacy classes of $G \wr \mathcal{S}_n$}\label{subsec:conj_class_wreath} The \textit{wreath product} $G \wr \mathcal{S}_n$ is the group with underlying set $G^n\times \mathcal{S}_n$ and product defined as follows:
$$((\sigma_1,\ldots ,\sigma_n); p)\cdot ((\epsilon_1,\ldots ,\epsilon_n); q)=((\sigma_{q^{-1}(1)}\epsilon_1,\ldots ,\sigma_{q^{-1}(1)}\epsilon_n);pq),$$
for any $((\sigma_1,\ldots ,\sigma_n); p),((\epsilon_1,\ldots ,\epsilon_n); q)\in G^n\times \mathcal{S}_n.$ We apply $p$ before $q$ when we write the product $pq.$ The identity in this group is $(1;1):=((1_G,1_G,\ldots ,1_G); \Id_n).$ The inverse of an element $((\sigma_1,\sigma_2,\ldots ,\sigma_n); p)\in G \wr \mathcal{S}_n$ is given by $$((\sigma_1,\sigma_2,\ldots ,\sigma_n); p)^{-1}=((\sigma^{-1}_{p(1)},\sigma^{-1}_{p(2)},\ldots ,\sigma^{-1}_{p(n)}); p^{-1}).$$

Let $x = (g; p)\in G \wr \mathcal{S}_n,$ where $g = (g_1,\ldots , g_n) \in G^n$ and $p \in \mathcal{S}_n$ is written as a product of disjoint cycles. If $(i_1,i_2,\ldots, i_r)$ is a cycle of $p,$ the element $g_{i_r}g_{i_{r-1}}\ldots g_{i_1}\in G$ is determined up to conjugacy in $G$ by $g$ and $(i_1,i_2,\ldots, i_r),$ and is called the \textit{cycle product} of $x$ corresponding to the cycle $(i_1,i_2,\ldots, i_r),$ see \cite[Page $170$]{McDo}. For any conjugacy class $c\in G_\star,$ we denote by $\rho(c)$ the partition written in the exponential way where $m_i(\rho(c))$ is the number of cycles of length $i$ in $p$ whose cycle-product lies in $c$ for each integer $i\geq 1.$ Then each element $x=(g; p)\in G \wr \mathcal{S}_n$ gives rise to a family of partitions $(\rho(c))_{c\in G_\star}$ indexed by $G_\star$ 
such that 
$$\sum_{i\geq 1,c\in G_\star}im_i(\rho(c))=n.$$ 
This family of partitions is called the \textit{type} of $x$ and denoted $\ty(x).$ 
\begin{ex}\label{main_ex} When $G=\mathbb{Z}_k,$ the type of $x= (g; p)\in \mathbb{Z}_k \wr \mathcal{S}_n$ is a $k$-vector of partitions $\Lambda=(\lambda_0,\lambda_1,\ldots, \lambda_{k-1})$ where each partition $\lambda_i$ is formed out of cycles $c$ of $p$ whose cycle product equals $i.$ 
 For example, consider the element $x=(g,p)\in \mathbb{Z}_{3}\wr \mathcal{S}_{10}$ where $g=(1,0,2,0,0,1,1,2,1,0)$ and $p=(1,4)(2,5)(3)(6)(7,8,9,10).$ The cycle product of $(1,4)$ is $1+0=1,$ of $(2,5)$ is $0+0=0,$ of $(3)$ is $2,$ of $(6)$ is $1$ and of $(7,8,9,10)$ is $1+2+1+0=1$ in $\mathbb{Z}_{3}.$ Thus $\ty(x)=(\lambda_0,\lambda_1,\lambda_2)$ with $\lambda_0=(2),$ $\lambda_1=(4,2,1)$ and $\lambda_2=(1).$
\end{ex}

It turns out, see \cite[Page $170$]{McDo}, that two permutations are conjugate in $G\wr \mathcal{S}_n$ if and only if they have the same type. Thus the conjugacy classes of $G\wr \mathcal{S}_n$ can be indexed by the elements of $\mathcal{P}_n^{G_\star}.$ If $\Lambda\in \mathcal{P}_n^{G_\star},$ we will denote by $C_\Lambda$ its associated conjugacy class:
\begin{equation*}
C_\Lambda:=\lbrace x\in G\wr \mathcal{S}_n; \ty(x)=\Lambda\rbrace.
\end{equation*}
From \cite[(3.1)]{McDo}, the order of the centralizer of an element of type $\Lambda$ in $G\wr \mathcal{S}_n$ is 
\begin{equation*}
Z_\Lambda=\prod_{c\in G_\star}z_{\Lambda(c)}\xi_c^{l(\Lambda(c))},
\end{equation*}
where $\xi_c=\frac{|G|}{|c|}$ is the order of the centralizer of an element $g\in c$ in $G.$ Thus, if $\Lambda\in \mathcal{P}_n^{G_\star},$ the cardinal of $C_\Lambda$ is given by:
\begin{equation*}
|C_\Lambda|=\frac{|G\wr \mathcal{S}_n|}{Z_\Lambda}=\frac{|G|^nn!}{\prod\limits_{c\in G_\star}z_{\Lambda(c)}\xi_c^{l(\Lambda(c))}}.
\end{equation*} 

\subsection{The center of the group $G\wr \mathcal{S}_n$ algebra} The \textit{group algebra} of $G\wr \mathcal{S}_n,$ denoted by $\mathbb{C}[G\wr \mathcal{S}_n],$ is the algebra over $\mathbb{C}$ with basis the elements of the group $G\wr \mathcal{S}_n.$ The product in $\mathbb{C}[G\wr \mathcal{S}_n]$ is the linear extension of the group product in $G\wr \mathcal{S}_n.$ The \textit{center} of the group algebra $\mathbb{C}[G\wr \mathcal{S}_n],$ usually denoted by $Z(\mathbb{C}[G\wr \mathcal{S}_n]),$ is the sub-algebra of $\mathbb{C}[G\wr \mathcal{S}_n]$ of invariant elements
under the conjugation action of $G\wr \mathcal{S}_n$ on $\mathbb{C}[G\wr \mathcal{S}_n]:$
$$Z(\mathbb{C}[G\wr \mathcal{S}_n]) := \lbrace x \in  \mathbb{C}[G\wr \mathcal{S}_n] ; yx = xy~~~~ \forall y \in  G\wr \mathcal{S}_n\rbrace.$$
The conjugacy classes of $G\wr \mathcal{S}_n$ index a basis of $Z(\mathbb{C}[G\wr \mathcal{S}_n]).$ We showed in Section \ref{subsec:conj_class_wreath} that the conjugacy classes of $G\wr \mathcal{S}_n$ are indexed by the elements of $\mathcal{P}_n^{G_\star}.$ Thus, the family $({\bf C}_\Lambda)_{\Lambda\in \mathcal{P}_n^{G_\star}},$ where
\begin{equation*}
{\bf C}_\Lambda=\sum_{x\in C_\Lambda}x,
\end{equation*}
forms a linear basis for $Z(\mathbb{C}[G\wr \mathcal{S}_n]).$ Let $\Lambda$ and $\Delta$ be two elements of $\mathcal{P}_n^{G_\star},$ the \textit{structure coefficients} $c_{\Lambda\Delta}^{\Gamma}$ of the algebra $Z(\mathbb{C}[G\wr \mathcal{S}_n])$ are defined by the following equation:
\begin{equation}\label{eq_str_coeff}
\mathbf{C}_\Lambda \mathbf{C}_\Delta=\sum_{\Gamma \in \mathcal{P}_n^{G_\star}} c_{\Lambda\Delta}^{\Gamma} \mathbf{C}_\Gamma.
\end{equation}
The coefficients $c_{\Lambda\Delta}^{\Gamma}$ are non-negative integer since they count the number of pairs of elements $(x, y)\in C_\Lambda\times C_\Delta$ such that $x\cdot y = z$ for a fixed element $z\in C_\Gamma.$ However, it is a very hard problem to compute these coefficients even in particular cases. For instance, the easiest choice for $G$ is the trivial group in which the group $G\wr \mathcal{S}_n$ is isomorphic to $\mathcal{S}_n.$ There is no explicit formula to compute all the structure coefficients of the center of the symmetric group algebra. Explicit formulas for particular structure coefficients of $Z(\mathbb{C}[\mathcal{S}_n])$ appeared in many papers, for example see \cite{katriel1987explicit}, \cite{GoupilSchaefferStructureCoef} and \cite{Toutjacodesmath}.

 In \cite{FaharatHigman1959}, Farahat and Higman showed that the structure coefficients of $Z(\mathbb{C}[\mathcal{S}_n])$ are polynomials in $n$ which was later proved by Ivanov and Kerov in \cite{Ivanov1999} using a more combinatorial way. In \cite{wang2004farahat}, following the Farahat and Higman approach, Wang proved that the structure coefficients $c_{\Lambda\Delta}^{\Gamma}$ of $Z(\mathbb{C}[G\wr \mathcal{S}_n])$ are polynomials in $n.$ In the next sections we will develop a combinatorial approach in order to prove Wang's result using the Ivanov-Kerov method.

\section{$G$-partial permutations}\label{sec_3} 
A \textit{partial permutation} of $[n]$ is a pair $(d, \omega)$ consisting of an arbitrary subset $d$ of $[n]$ and an arbitrary bijection $\omega : d \longrightarrow d.$ The notion of partial permutation of $[n]$ appeared in \cite{Ivanov1999} to show by a combinatorial way that the structure coefficients of the center of the symmetric group $\mathcal{S}_n$ algebra are polynomials in $n.$

If $d$ is a subset of $[n],$ we denote by $G^n_d$ the set of vectors $g$ with $n$-coordinates such that $g_i\in G$ if $i\in d$ and $g_i$ is left blank otherwise. For example, if $G=\mathbb{Z}_3,$ $n=5$ and $d=\lbrace 1,3,4\rbrace$ then $(1,,2,0,)\in G^5_d$ but $(1,,1,1,1)\notin G^5_d.$
\begin{definition}
A $G$-partial permutation of $[n]$ is a pair $(g; (d,\omega))$ where $(d,\omega)$ is a partial permutation of $[n]$ and $g  \in G^n_d.$ 
\end{definition}
We denote by $\mathfrak{P}^G_n$ the set of all $G$-partial permutations of $[n].$ It would be clear that 
\begin{equation*}
\vert \mathfrak{P}^G_n\vert=\sum_{k=0}^n{n \choose k}k!|G|^k.
\end{equation*}

A $G$-partial permutation $(g; (d,\omega))$ of $[n]$ may be represented by a diagram obtained by drawing the two lines permutation diagram associated to $(d,\omega),$ with the nodes of the bottom row replaced by the elements $g_i$ for $i\in d.$ This representation will help us understanding the product between $G$-partial permutations of $[n]$ which will be defined later. 
\begin{ex}\label{ex:2.3} If $n=9,$ $A=\lbrace 2,4,5,6\rbrace$ and $\omega= (2, 5) (4,6),$ then we represent the element $(g; (A,\omega))$ by the following diagram
$$
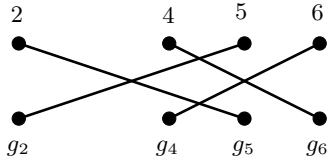
\begin{tikzpicture}[line cap=round,line join=round,>=triangle 45,x=.5cm,y=.5cm]
\draw [line width=1pt] (-20,4)-- (-14,2);
\draw [line width=1pt] (-14,4)-- (-20,2);
\draw [line width=1pt] (-16,4)-- (-12,2);
\draw [line width=1pt] (-12,4)-- (-16,2);
\begin{scriptsize}
\draw [fill=black] (-20,4) circle (2.5pt);
\draw[color=black] (-20.044879600000034,4.78772291818183) node {$2$};
\draw [fill=black] (-16,4) circle (2.5pt);
\draw[color=black] (-16.01860460000003,4.75551271818183) node {$4$};
\draw [fill=black] (-14,4) circle (2.5pt);
\draw[color=black] (-14.085992600000028,4.85214331818183) node {$5$};
\draw [fill=black] (-12,4) circle (2.5pt);
\draw[color=black] (-12.05675,4.81993311818183) node {$6$};
\draw [fill=black] (-14,2) circle (2.5pt);
\draw[color=black] (-14.021572200000026,1.2) node {$g_5$};
\draw [fill=black] (-20,2) circle (2.5pt);
\draw[color=black] (-20.012669400000036,1.2) node {$g_2$};
\draw [fill=black] (-12,2) circle (2.5pt);
\draw[color=black] (-12.05675,1.2) node {$g_6$};
\draw [fill=black] (-16,2) circle (2.5pt);
\draw[color=black] (-16.05081480000003,1.2) node {$g_4$};
\end{scriptsize}
\end{tikzpicture}$$
\end{ex}
The definition of type can be extended naturally to a $G$-partial permutation of $[n].$ If $x=(g; (d,\omega))$ is a $G$-partial permutation of $[n]$ and $c\in G_\star$ then let $\rho(c)$ be the partition written in the exponential way where $m_i(\rho(c))$ is the number of cycles of length $i$ in $\omega$ whose cycle-product lies in $c$ for each integer $i\geq 1.$ Define the type of $x$ to be the family of partitions $(\rho(c))_{c\in G_\star}$ indexed by $G_\star.$ It would be clear that 
$$|\ty(x)|=|d|.$$ 

If $x=(g; (d,\omega))\in \mathfrak{P}^G_n,$ we denote by $\widetilde{x}$ the element $(\widetilde{g};\widetilde{\omega})$ of $G \wr \mathcal{S}_n$ where $\widetilde{\omega}$ and $\widetilde{g}$ are defined by:
$$\widetilde{\omega}(a)=
\left\{
\begin{array}{ll}
  \omega(a) & \qquad \mathrm{if}\quad a\in d, \\
  a & \qquad \mathrm{if}\quad a\in [n]\setminus d. \\
 \end{array}
  \right. \text{ and } \widetilde{g}_i=
\left\{
\begin{array}{ll}
  g_i & \qquad \mathrm{if}\quad i\in d, \\
  1_G & \qquad \mathrm{if}\quad i\in [n]\setminus d. \\
 \end{array}
 \right.$$
The product of two $G$-partial permutations $(g; (d_1,\omega_1))$ and $(h; (d_2,\omega_2))$ of $[n]$ is defined by:
\begin{equation*}\big(g; (d_1,\omega_1)\big)\cdot \big(h; (d_2,\omega_2)\big)=\big((\widetilde{g}_{\widetilde{\omega}_{2_{|d_1\cup d_2}}^{-1}(1)}\widetilde{h}_1,\ldots ,\widetilde{g}_{\widetilde{\omega}_{2_{|d_1\cup d_2}}^{-1}(n)}\widetilde{h}_n);(d_1,\omega_1) (d_2,\omega_2)\big),\end{equation*}
where
\begin{equation*}(d_1,\omega_1) (d_2,\omega_2)=(d_1\cup d_2,\widetilde{\omega}_{1_{|d_1\cup d_2}} \widetilde{\omega}_{2_{|d_1\cup d_2}}).\end{equation*}
This product is well defined since $(d_1,\omega_1) (d_2,\omega_2)$ is a partial permutation of $[n].$ The set $\mathfrak{P}^G_n$ is a semigroup with this multiplication. The unity in $\mathfrak{P}^G_n$ is the $G$-partial permutation $((,,\ldots ,);(\emptyset,e_0))$ where $e_0$ is the trivial permutation of the empty set $\emptyset.$ We denote by $\mathcal{B}^G_n=\mathbb{C}[\mathfrak{P}^G_n]$ the algebra of the semigroup $\mathfrak{P}^G_n.$ 

\begin{ex} Reconsider the $G$-partial permutation $(g; (A,\omega))$ of Example \ref{ex:2.3} and let $(f; (B,\sigma))$ be the $G$-partial permutation of $[9]$ with $B=\lbrace 1,3,5,6,8,9\rbrace$ and $\sigma=(1,5,8)(3,9)(6)$ then the product $(g; (A,\omega))\cdot (f; (B,\sigma))$ yields the following $G$-partial permutation of $[9]$
$$\big((f_1,g_2,f_3,g_4,f_5,g_6f_6,,g_5f_8,f_9);(\lbrace 1,2,3,4,5,6,8,9\rbrace,(1,5,2,8)(3,9)(4,6))\big)$$
This can be obtained easily by drawing the diagram of $(g; (A,\omega))$ above the diagram of $(f; (B,\sigma))$ then extending both of them to $A\cup B$ as represented below
$$
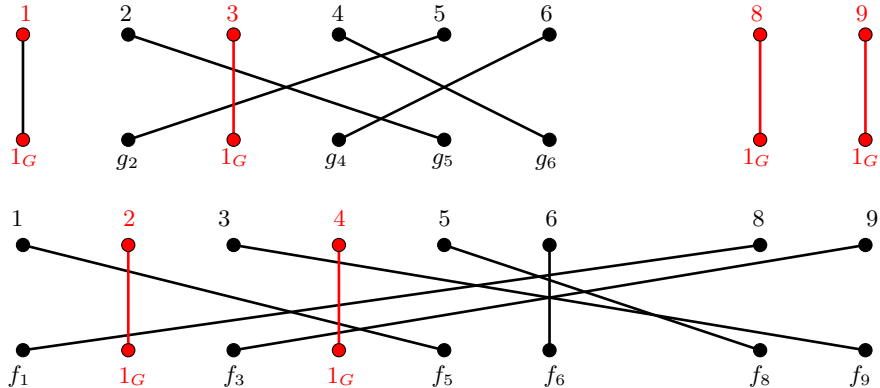
\begin{tikzpicture}[line cap=round,line join=round,>=triangle 45,x=0.7cm,y=0.7cm]
\draw [line width=1pt] (-20,4)-- (-14,2);
\draw [line width=1pt] (-14,4)-- (-20,2);
\draw [line width=1pt] (-16,4)-- (-12,2);
\draw [line width=1pt] (-12,4)-- (-16,2);
\draw [line width=1pt] (-22,4)-- (-22,2);
\draw [red, line width=1pt] (-18,4)-- (-18,2);
\draw [red, line width=1pt] (-8,4)-- (-8,2);
\draw [red, line width=1pt] (-6,4)-- (-6,2);
\draw [line width=1pt] (-22,0)-- (-14,-2);
\draw [line width=1pt] (-14,0)-- (-8,-2);
\draw [line width=1pt] (-8,0)-- (-22,-2);
\draw [line width=1pt] (-18,0)-- (-6,-2);
\draw [line width=1pt] (-6,0)-- (-18,-2);
\draw [line width=1pt] (-12,0)-- (-12,-2);
\draw [red, line width=1pt] (-20,0)-- (-20,-2);
\draw [red, line width=1pt] (-16,0)-- (-16,-2);
\begin{scriptsize}
\draw [fill=black] (-20,4) circle (2.5pt);
\draw[color=black] (-20.044879600000037,4.4) node {$2$};
\draw [fill=black] (-16,4) circle (2.5pt);
\draw[color=black] (-16.018604600000035,4.4) node {$4$};
\draw [fill=black] (-14,4) circle (2.5pt);
\draw[color=black] (-14.085992600000031,4.4) node {$5$};
\draw [fill=black] (-12,4) circle (2.5pt);
\draw[color=black] (-12.05675,4.4) node {$6$};
\draw [fill=black] (-14,2) circle (2.5pt);
\draw[color=black] (-14.02157220000003,1.5989131181818232) node {$g_5$};
\draw [fill=black] (-20,2) circle (2.5pt);
\draw[color=black] (-20.01266940000004,1.5667029181818233) node {$g_2$};
\draw [fill=black] (-12,2) circle (2.5pt);
\draw[color=black] (-12.05675,1.5667029181818233) node {$g_6$};
\draw [fill=black] (-16,2) circle (2.5pt);
\draw[color=black] (-16.050814800000033,1.5989131181818232) node {$g_4$};
\draw [fill=red] (-22,4) circle (2.5pt);
\draw[color=red] (-21.945281400000038,4.4) node {$1$};
\draw [fill=red] (-22,2) circle (2.5pt);
\draw[color=red] (-21.97749160000004,1.6633335181818234) node {$1_G$};
\draw [fill=red] (-18,4) circle (2.5pt);
\draw[color=red] (-18.015637000000037,4.4) node {$3$};
\draw [fill=red] (-18,2) circle (2.5pt);
\draw[color=red] (-17.983426800000036,1.6633335181818234) node {$1_G$};
\draw [fill=red] (-8,4) circle (2.5pt);
\draw[color=red] (-8.062685200000026,4.4) node {$8$};
\draw [fill=red] (-8,2) circle (2.5pt);
\draw[color=red] (-8.062685200000026,1.6633335181818234) node {$1_G$};
\draw [fill=red] (-6,4) circle (2.5pt);
\draw[color=red] (-6.065652800000024,4.4) node {$9$};
\draw [fill=red] (-6,2) circle (2.5pt);
\draw[color=red] (-6.001232400000024,1.6633335181818234) node {$1_G$};
\draw [fill=black] (-22,0) circle (2.5pt);
\draw[color=black] (-22.10633240000004,0.5) node {$1$};
\draw [fill=black] (-14,-2) circle (2.5pt);
\draw[color=black] (-14.02157220000003,-2.5) node {$f_5$};
\draw [fill=black] (-14,0) circle (2.5pt);
\draw[color=black] (-13.98936200000003,0.5) node {$5$};
\draw [fill=black] (-8,-2) circle (2.5pt);
\draw[color=black] (-7.998264800000025,-2.5) node {$f_8$};
\draw [fill=black] (-8,0) circle (2.5pt);
\draw[color=black] (-8.030475000000026,0.5) node {$8$};
\draw [fill=black] (-22,-2) circle (2.5pt);
\draw[color=black] (-22.07412220000004,-2.5) node {$f_1$};
\draw [fill=black] (-18,0) circle (2.5pt);
\draw[color=black] (-18.176688000000034,0.5) node {$3$};
\draw [fill=black] (-6,-2) circle (2.5pt);
\draw[color=black] (-6.097863000000024,-2.5) node {$f_9$};
\draw [fill=black] (-6,0) circle (2.5pt);
\draw[color=black] (-5.872391600000024,0.5) node {$9$};
\draw [fill=black] (-18,-2) circle (2.5pt);
\draw[color=black] (-17.983426800000036,-2.5) node {$f_3$};
\draw [fill=black] (-12,0) circle (2.5pt);
\draw[color=black] (-11.96011940000003,0.5) node {$6$};
\draw [fill=black] (-12,-2) circle (2.5pt);
\draw[color=black] (-11.895699000000029,-2.5) node {$f_6$};
\draw [fill=red] (-20,0) circle (2.5pt);
\draw[color=red] (-19.980459200000038,0.5) node {$2$};
\draw [fill=red] (-20,-2) circle (2.5pt);
\draw[color=red] (-19.916038800000038,-2.5) node {$1_G$};
\draw [fill=red] (-16,0) circle (2.5pt);
\draw[color=red] (-15.986394400000034,0.5) node {$4$};
\draw [fill=red] (-16,-2) circle (2.5pt);
\draw[color=red] (-15.954184200000032,-2.5) node {$1_G$};
\end{scriptsize}
\end{tikzpicture}
$$
All extensions are drawn in red. The diagram of the product $(g; (A,\omega))\cdot(f; (B,\sigma))$ is then obtained by taking the resulted diagram of the above combination
$$
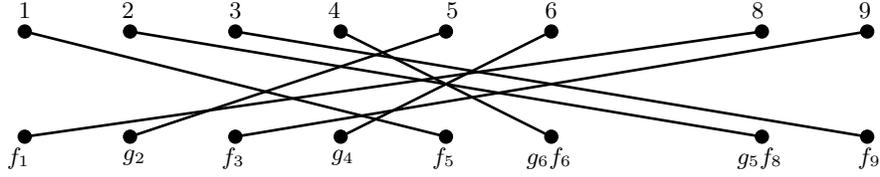
\begin{tikzpicture}[line cap=round,line join=round,>=triangle 45,x=1.4cm,y=1.4cm]
\draw [line width=1pt] (-6,0)-- (-2,-1);
\draw [line width=1pt] (-5,0)-- (1,-1);
\draw [line width=1pt] (-4,0)-- (2,-1);
\draw [line width=1pt] (-3,0)-- (-1,-1);
\draw [line width=1pt] (-2,0)-- (-5,-1);
\draw [line width=1pt] (-1,0)-- (-3,-1);
\draw [line width=1pt] (1,0)-- (-6,-1);
\draw [line width=1pt] (2,0)-- (-4,-1);
\begin{scriptsize}
\draw [fill=black] (-6,0) circle (2.5pt);
\draw[color=black] (-6,0.2) node {$1$};
\draw [fill=black] (-2,-1) circle (2.5pt);
\draw[color=black] (-2.02,-1.2) node {$f_5$};
\draw [fill=black] (-5,0) circle (2.5pt);
\draw[color=black] (-5.02,0.2) node {$2$};
\draw [fill=black] (1,-1) circle (2.5pt);
\draw[color=black] (0.98,-1.2) node {$g_5f_8$};
\draw [fill=black] (-4,0) circle (2.5pt);
\draw[color=black] (-4,0.2) node {$3$};
\draw [fill=black] (2,-1) circle (2.5pt);
\draw[color=black] (2.02,-1.2) node {$f_9$};
\draw [fill=black] (-3,0) circle (2.5pt);
\draw[color=black] (-3.06,0.2) node {$4$};
\draw [fill=black] (-1,-1) circle (2.5pt);
\draw[color=black] (-1.02,-1.2) node {$g_6f_6$};
\draw [fill=black] (-2,0) circle (2.5pt);
\draw[color=black] (-1.94,0.2) node {$5$};
\draw [fill=black] (-5,-1) circle (2.5pt);
\draw[color=black] (-4.96,-1.2) node {$g_2$};
\draw [fill=black] (-1,0) circle (2.5pt);
\draw[color=black] (-1,0.2) node {$6$};
\draw [fill=black] (-3,-1) circle (2.5pt);
\draw[color=black] (-2.98,-1.2) node {$g_4$};
\draw [fill=black] (1,0) circle (2.5pt);
\draw[color=black] (0.96,0.2) node {$8$};
\draw [fill=black] (-6,-1) circle (2.5pt);
\draw[color=black] (-6.06,-1.2) node {$f_1$};
\draw [fill=black] (2,0) circle (2.5pt);
\draw[color=black] (1.98,0.2) node {$9$};
\draw [fill=black] (-4,-1) circle (2.5pt);
\draw[color=black] (-4.02,-1.2) node {$f_3$};
\end{scriptsize}
\end{tikzpicture}
$$
\end{ex}

\section{Action of $G\wr \mathcal{S}_n$ on $\mathfrak{P}^G_n$}\label{sec_4} Any element $(g;\sigma)$ of the wreath product $G\wr \mathcal{S}_n$ can be seen as a $G$-partial permutation of $[n]$ by identifying $\sigma$ with the partial permutation $([n],\sigma).$ 
The wreath product $G\wr \mathcal{S}_n$ acts on the semigroup $\mathfrak{P}^G_n$ by:
$$(g;\sigma)\cdot \big(h;(d,\omega)\big):=\big(f;(\sigma^{-1}(d),\sigma \omega \sigma^{-1})\big),$$
where $f_i=g_{\omega^{-1}(\sigma(i))}h_{\sigma(i)}g^{-1}_{\sigma(i)}$ if $i\in \sigma^{-1}(d)$ for any $(g;\sigma)\in G\wr \mathcal{S}_n$ and $\big(h;(d,\omega)\big)\in \mathfrak{P}^G_n.$ 
The orbits of this action will be called the conjugacy classes of $\mathfrak{P}^G_n.$ Two $G$-partial permutations $(h; (d_1,\omega_1))$ and $(f; (d_2,\omega_2))$ of $[n]$ are in the same conjugacy class if and only if there exists $(g;\sigma)\in G\wr \mathcal{S}_n$ such that $(g;\sigma)\cdot (h; (d_1,\omega_1))=(f; (d_2,\omega_2)),$ that is $d_2=\sigma^{-1}(d_1),$ $\omega_2=\sigma \omega_1 \sigma^{-1}$ and $f_i=g_{\omega_1^{-1}(\sigma(i))}h_{\sigma(i)}g^{-1}_{\sigma(i)}$ for any $i\in d_2.$ 
\begin{ex} Let $\sigma=(2,3,6)(1,4)(5,7,9)(8)\in \mathcal{S}_9,$ $d=\lbrace 2,4,5,6\rbrace$ and $\omega=(2,5)(4,6).$ To obtain that 
$$(g;\sigma)\cdot \big(h;(d,\omega)\big)=\big((g_6h_4g_4^{-1},,g_4h_6g_6^{-1},,,g_5h_2g_2^{-1},,,g_2h_5g_5^{-1});(\lbrace 1,3,6,9\rbrace,(1,3)(6,9))\big),$$
we have to draw the diagram of $(g;\sigma)$ restricted to $\sigma^{-1}(d)$ then below it the diagram of $\big(h;(d,\omega)\big)$ then below it the diagram of $(g;\sigma)^{-1}$ restricted to $d$ as shown below
$$
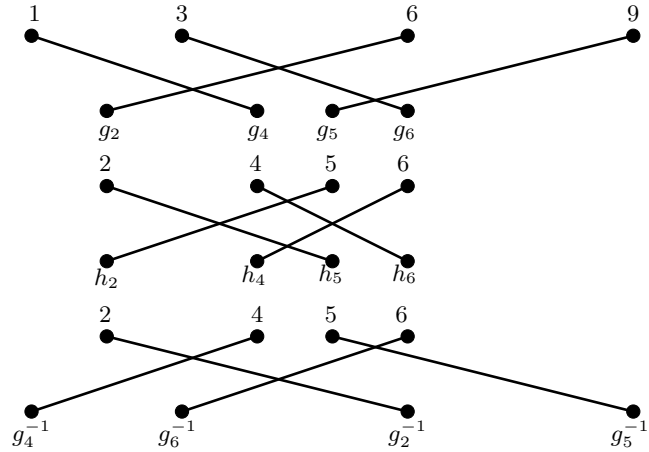
\begin{tikzpicture}[line cap=round,line join=round,>=triangle 45,x=1cm,y=1cm]
\draw [line width=1pt] (-4,-1)-- (-1,-2);
\draw [line width=1pt] (-2,-1)-- (0,-2);
\draw [line width=1pt] (-1,-1)-- (-4,-2);
\draw [line width=1pt] (0,-1)-- (-2,-2);
\draw [line width=1pt] (-4,0)-- (0,1);
\draw [line width=1pt] (-2,0)-- (-5,1);
\draw [line width=1pt] (-1,0)-- (3,1);
\draw [line width=1pt] (0,0)-- (-3,1);
\draw [line width=1pt] (-4,-3)-- (0,-4);
\draw [line width=1pt] (-2,-3)-- (-5,-4);
\draw [line width=1pt] (-1,-3)-- (3,-4);
\draw [line width=1pt] (0,-3)-- (-3,-4);
\begin{scriptsize}
\draw [fill=black] (-4,-1) circle (2.5pt);
\draw[color=black] (-4.02,-0.7) node {$2$};
\draw [fill=black] (-1,-2) circle (2.5pt);
\draw[color=black] (-1.02,-2.17) node {$h_5$};
\draw [fill=black] (-2,-1) circle (2.5pt);
\draw[color=black] (-2.02,-0.7) node {$4$};
\draw [fill=black] (0,-2) circle (2.5pt);
\draw[color=black] (-0.04,-2.17) node {$h_6$};
\draw [fill=black] (-1,-1) circle (2.5pt);
\draw[color=black] (-1.02,-0.7) node {$5$};
\draw [fill=black] (-4,-2) circle (2.5pt);
\draw[color=black] (-4,-2.23) node {$h_2$};
\draw [fill=black] (0,-1) circle (2.5pt);
\draw[color=black] (-0.06,-0.7) node {$6$};
\draw [fill=black] (-2,-2) circle (2.5pt);
\draw[color=black] (-2.04,-2.19) node {$h_4$};
\draw [fill=black] (-4,0) circle (2.5pt);
\draw[color=black] (-3.96,-0.3) node {$g_2$};
\draw [fill=black] (0,1) circle (2.5pt);
\draw[color=black] (0.06,1.3) node {$6$};
\draw [fill=black] (-2,0) circle (2.5pt);
\draw[color=black] (-1.98,-0.3) node {$g_4$};
\draw [fill=black] (-5,1) circle (2.5pt);
\draw[color=black] (-4.96,1.3) node {$1$};
\draw [fill=black] (-1,0) circle (2.5pt);
\draw[color=black] (-1.06,-0.3) node {$g_5$};
\draw [fill=black] (3,1) circle (2.5pt);
\draw[color=black] (3,1.3) node {$9$};
\draw [fill=black] (0,0) circle (2.5pt);
\draw[color=black] (-0.04,-0.3) node {$g_6$};
\draw [fill=black] (-3,1) circle (2.5pt);
\draw[color=black] (-3,1.3) node {$3$};
\draw [fill=black] (-4,-3) circle (2.5pt);
\draw[color=black] (-4.02,-2.7) node {$2$};
\draw [fill=black] (0,-4) circle (2.5pt);
\draw[color=black] (0,-4.3) node {$g_2^{-1}$};
\draw [fill=black] (-2,-3) circle (2.5pt);
\draw[color=black] (-2,-2.7) node {$4$};
\draw [fill=black] (-5,-4) circle (2.5pt);
\draw[color=black] (-5,-4.3) node {$g^{-1}_4$};
\draw [fill=black] (-1,-3) circle (2.5pt);
\draw[color=black] (-1.02,-2.7) node {$5$};
\draw [fill=black] (3,-4) circle (2.5pt);
\draw[color=black] (2.96,-4.3) node {$g_5^{-1}$};
\draw [fill=black] (0,-3) circle (2.5pt);
\draw[color=black] (-0.08,-2.7) node {$6$};
\draw [fill=black] (-3,-4) circle (2.5pt);
\draw[color=black] (-3.06,-4.3) node {$g_6^{-1}$};
\end{scriptsize}
\end{tikzpicture}
$$
Then it would be easy to verify that $\ty\Big((g;\sigma)\cdot \big(h;(d,\omega)\big)\Big)=\ty\Big(\big(h;(d,\omega)\big)\Big)$ since the cycle product $h_5h_2$ of the cycle $(2,5)$ of $\omega$ is conjugate to the cycle product $g_2h_5g_5^{-1}g_5h_2g_2^{-1}$ of the cycle $(6,9)$ of $\sigma\omega\sigma^{-1}$ and the cycle product $h_6h_4$ of the cycle $(4,6)$ of $\omega$ is conjugate to the cycle product $g_4h_6g_6^{-1}g_6h_4g_4^{-1}$ of the cycle $(1,3)$ of $\sigma\omega\sigma^{-1}.$
\end{ex}
This shows that the conjugacy classes of $\mathfrak{P}^G_n$ can be indexed by the elements of the set $\mathcal{P}^{G_\star}_{\leq n}$ of families of partitions indexed by $G_\star$ with size less than or equal to $n.$ If $\Lambda=(\Lambda(c))_{c\in G_\star}\in \mathcal{P}^{G_\star}_{\leq n},$ the conjugacy class of $\mathfrak{P}^G_n$ associated to $\Lambda$ will be denoted $C_{\Lambda;n}$ and is defined by:
$$C_{\Lambda;n}:=\lbrace x=(h;(d,\omega))\in \mathfrak{P}^G_n\text{ such that } |d|=|\Lambda| \text{ and }\ty(x)=\Lambda\rbrace.$$

\begin{prop}\label{prop_card_C} If $\Lambda=(\Lambda(c))_{c\in G_\star}\in \mathcal{P}^{G_\star}_{\leq n},$ then:
$$\vert C_{\Lambda;n}\vert=\begin{pmatrix}
n-|\Lambda|+m_1(\lambda(\lbrace 1_G\rbrace)) \\
m_1(\lambda(\lbrace 1_G\rbrace))
\end{pmatrix} \vert C_{\underline{\Lambda}_n}\vert.$$
\end{prop}
\begin{proof} 
Consider the following mapping $$\begin{array}{ccccc}
\Theta & : & C_{\Lambda;n} & \to & C_{\underline{\Lambda}_n} \\
& & (g;(d,\omega)) & \mapsto & (\widetilde{g};\widetilde{\omega}). \\
\end{array}$$
If $v,v^{'}\in C_{\underline{\Lambda}_n}$ with $v\neq v^{'},$ we have $\Theta^{-1}(v)\cap \Theta^{-1}(v^{'})=\emptyset$ which implies that 
\begin{equation}\label{int_eq_1}
\vert C_{\Lambda;n}\vert=\sum_{v\in C_{\underline{\Lambda}_n}}\vert \Theta^{-1}(v)\vert.
\end{equation}
Let $(\sigma;p)\in C_{\underline{\Lambda}_n}$ and consider the set $\supp(\sigma;p)$ defined as follows:
$$\supp(\sigma;p)=\lbrace i\in [n] \text{ such that } p(i)\neq i \text{ or } p(i)=i \text{ and } \sigma_i\neq 1_G\rbrace.$$
Since $(\sigma;p)$ has type $\underline{\Lambda}_n$ then it would be clear that $\vert \supp(\sigma;p)\vert=\vert\Lambda\vert-m_1(\lambda(\lbrace 1_G\rbrace)).$ 
To make an element $(g;(d,\omega))\in \Theta^{-1}(\sigma;p),$  $(g;(d,\omega))$ must coincide with $(\sigma;p)$ on $\supp(\sigma;p)$ which necessarily implies that $\supp(\sigma;p)\subset d$ and we choose $m_1(\Lambda(\lbrace 1_G\rbrace))=\vert d\vert-\vert \supp(\sigma;p)\vert$ fixed points for $(d,\omega)$ among the $n-|\Lambda|+m_1(\Lambda(\lbrace 1_G\rbrace))$ fixed points of $\sigma.$ Thus for any $(\sigma;p)\in C_{\underline{\Lambda}_n}$ we have $$\vert \Theta^{-1}(\sigma;p)\vert=\begin{pmatrix}
n-|\Lambda|+m_1(\lambda(\lbrace 1_G\rbrace)) \\
m_1(\lambda(\lbrace 1_G\rbrace))
\end{pmatrix}.$$
Combining this formula with Equation (\ref{int_eq_1}) ends the proof.
\end{proof}

We extend the action of $G\wr \mathcal{S}_n$ on $\mathfrak{P}^G_n$ by linearity to an action on $\mathcal{B}^G_n:=\mathbb{C}[\mathfrak{P}^G_n],$ the algebra of the semigroup $\mathfrak{P}^G_n,$ and we denote  
$$\mathcal{A}^G_n:=\lbrace b\in \mathcal{B}^G_n \text{ such that for any $(\sigma;p)\in G\wr \mathcal{S}_n,$ } (\sigma;p)\cdot b=b \rbrace$$ 
the sub-algebra of invariant elements under this action. 
\begin{prop}\label{prop_psi}
The surjective homomorphism $$\begin{array}{ccccc}
\psi & : & \mathfrak{P}^G_n & \to & G\wr\mathcal{S}_n \\
& & (h;(d,\omega)) & \mapsto & (\widetilde{h};\widetilde{\omega}) \\
\end{array}$$
is compatible with the action of $G\wr \mathcal{S}_n$ on $\mathfrak{P}^G_n.$
\end{prop}
\begin{proof}
We need to prove that for any $(g;\sigma)\in G\wr \mathcal{S}_n$ and any $\big(h;(d,\omega)\big)\in \mathfrak{P}^G_n$ we have:
$$\psi \big((g;\sigma)\cdot \big(h;(d,\omega)\big)\big)=(g;\sigma)\cdot \psi\big(\big(h;(d,\omega)\big)\big)=(g;\sigma)(\widetilde{h};\widetilde{\omega})(g;\sigma)^{-1}.$$
We have $(g;\sigma)\cdot \big(h;(d,\omega)\big)=\big(f;(\sigma^{-1}(d),\sigma \omega \sigma^{-1})\big)$
where $f_i=g_{\omega^{-1}(\sigma(i))}h_{\sigma(i)}g^{-1}_{\sigma(i)}$ if $i\in \sigma^{-1}(d)$ which implies that $\psi \big((g;\sigma)\cdot \big(h;(d,\omega)\big)\big)=\big(\widetilde{f};\widetilde{\sigma \omega \sigma^{-1}}\big)$ with
$$\widetilde{\sigma \omega \sigma^{-1}}(a)=
\left\{
\begin{array}{ll}
  (\sigma \omega \sigma^{-1})(a) & \qquad \mathrm{if}\quad a\in \sigma^{-1}(d) \\
  a & \qquad \mathrm{otherwise} \\
 \end{array}
  \right. \text{ and } \widetilde{f}_i=
\left\{
\begin{array}{ll}
  f_i & \qquad \mathrm{if}\quad i\in \sigma^{-1}(d) \\
  1_G & \qquad \mathrm{otherwise}. \\
 \end{array}
 \right.$$
On the other hand $(g;\sigma)(\widetilde{h};\widetilde{\omega})(g;\sigma)^{-1}=(r; \sigma\widetilde{\omega}\sigma^{-1})$ with $r_i=g_{\widetilde{\omega}^{-1}(\sigma(i))}\widetilde{h}_{\sigma(i)}g^{-1}_{\sigma(i)}.$ If $i\notin \sigma^{-1}(d)$ then $\sigma(i)\notin d$ which implies that $\widetilde{\omega}^{-1}(\sigma(i))=\sigma(i)$ and $\widetilde{h}_{\sigma(i)}=1_G$ which results in $r_i=1_G.$ If $i\in \sigma^{-1}(d)$ then $r_i=f_i.$ Thus $\widetilde{f}=r$ and it is easy to check that $\widetilde{\sigma \omega \sigma^{-1}}=\sigma\widetilde{\omega}\sigma^{-1}$ which ends the proof.
\end{proof}
The surjective homomorphism $\psi$ can be extended to a surjective homomorphism of algebras $\psi:\mathcal{B}^G_n\rightarrow \mathbb{C}[G\wr \mathcal{S}_n]$ and Proposition \ref{prop_psi} implies that
$$\psi(\mathcal{A}^G_n)=Z(\mathbb{C}[\mathcal{S}_n]).$$

If $\Lambda=(\Lambda(c))_{c\in G_\star}\in \mathcal{P}^{G_\star}_{\leq n},$ let us denote by ${\bf{C}}_{\Lambda;n}$ the following formal sum:
$${\bf C}_{\Lambda;n}=\sum_{(h;(d,\omega))\in C_{\Lambda;n}}(h;(d,\omega)).$$
The elements of the family $({\bf{C}}_{\Lambda;n})_{\Lambda\in \mathcal{P}^{G_\star}_{\leq n}}$ form a basis for the algebra $\mathcal{A}^G_n$ and if $\Lambda\in \mathcal{P}^{G_\star}_{\leq n},$ then by Proposition \ref{prop_card_C} we have:
\begin{equation}\label{imagepsi}
\psi({\bf{C}}_{\Lambda;n})=\begin{pmatrix}
n-|\Lambda|+m_1(\lambda(\lbrace 1_G\rbrace)) \\
m_1(\lambda(\lbrace 1_G\rbrace))
\end{pmatrix} {\bf C}_{\underline{\Lambda}_n}.
\end{equation}

\section{Action of $G\wr \mathcal{S}_\infty$ on $\mathcal{B}^G_\infty$}\label{sec_5}
Denote by $\mathcal{B}^G_\infty$ the projective limit of the algebras $\mathcal{B}^G_n$ with respect to the homomorphism $\varphi_n:\mathcal{B}^G_{n+1}\rightarrow \mathcal{B}^G_n$ defined by:
$$\varphi_n(h;(d,\omega))=
\left\{
\begin{array}{ll}
  (h;(d,\omega)) & \qquad \mathrm{if}\quad d\subset [n], \\
  0 & \qquad \mathrm{otherwise.}\quad \\
 \end{array}
 \right.$$
If $d$ is a finite subset of $\mathbb{N},$ we define $G^\infty_d$ to be the set of vectors $x$ with infinite number of coordinates such that $x_i\in G$ whenever $i\in d$ and $x_i$ is left blank otherwise. An element $b\in \mathcal{B}^G_\infty$ can be canonically written:
$$b=\sum_{k=0}^{\infty}\sum_{\vert d\vert=k}\sum_{\omega\in \mathcal{S}_d}\sum_{h\in G^\infty_d}b_{(h;(d,\omega))}(h;(d,\omega)),$$
where the $b_{(h;(d,\omega))}$'s are complex numbers. Consider the group $\mathcal{S}_\infty$ of permutations of $\mathbb{N}$ with finite support. Let the group $G\wr \mathcal{S}_\infty$ acts on $\mathcal{B}^G_\infty$ by conjugation and  denote by $\mathcal{A}^G_\infty$ the sub-algebra of invariant elements of $\mathcal{B}^G_\infty$ under this action. An element $b\in \mathcal{B}^G_\infty$ is in $\mathcal{A}^G_\infty$ if and only if:
$$b_{(h;(d,\omega))}=b_{(f;(\sigma^{-1}(d),\sigma\omega \sigma^{-1}))} \text{ for any } (g;\sigma)\in G\wr \mathcal{S}_\infty,$$
where $f_i=g_{\omega_1^{-1}(\sigma(i))}h_{\sigma(i)}g^{-1}_{\sigma(i)}$ if $i\in \sigma^{-1}(d)$ and $f_i$ is left blank if   $i\notin \sigma^{-1}(d).$

\begin{definition}
A $G$-partial permutation of $\mathbb{N}$ is a pair $(h;(d,\omega))$ where $d\subsetneq \mathbb{N}$ is a finite subset of $\mathbb{N},$ $\omega\in \mathcal{S}_d$ and $h\in G_d^\infty.$
\end{definition} 
We denote by $\mathfrak{P}_\infty^G$ the set of all $G$-partial permutations of $\mathbb{N}.$
For a family of partitions $\Lambda=(\lambda(c))_{c\in G_\star}\in \mathcal{P}^{G_\star}$ indexed by $G_\star,$ we define $C_{\Lambda;\infty}$ as follows:
$$C_{\Lambda;\infty}=\lbrace (h;(d,\omega))\in \mathfrak{P}_\infty^G\text{ such that }|d|=|\Lambda|\text{ and }\ty(h;(d,\omega))=\Lambda\rbrace.$$
Any element in $\mathcal{A}^G_\infty$ can be written in a unique way as an infinite linear combination of elements $({\bf C}_{\Lambda;\infty})_{\Lambda \in \mathcal{P}^G},$ where
$${\bf C}_{\Lambda;\infty}=\sum_{(h;(d,\omega))\in C_{\Lambda;\infty}}(h;(d,\omega)). $$
Let $\Lambda$ and $\Delta$ be two families of partitions in $\mathcal{P}^{G_\star},$ the structure coefficients $k_{\Lambda\Delta}^\Gamma$ of the algebra $\mathcal{A}^G_\infty$ are defined by:
\begin{equation}\label{eq:str_coef_A_infini}
{\bf C}_{\Lambda;\infty}{\bf C}_{\Delta;\infty}=\sum_{\Gamma \in \mathcal{P}^{G_\star}}k_{\Lambda\Delta}^\Gamma{\bf C}_{\Gamma;\infty}.
\end{equation}

\begin{prop} The function $\Fil$ defined by $\Fil({\bf C}_{\Lambda;\infty})=|\Lambda|,$ for any $\Lambda \in \mathcal{P}^{G_\star},$ is a filtration on $\mathcal{A}^G_\infty.$
\end{prop}

\begin{proof} We need to prove that $\Fil({\bf C}_{\Lambda;\infty}{\bf C}_{\Delta;\infty})\leq \Fil({\bf C}_{\Lambda;\infty})+\Fil({\bf C}_{\Delta;\infty})$ for any two families of partitions $\Lambda,\Delta\in \mathcal{P}^{G_\star}.$ For this let $\Gamma$ be a family of partitions in $\mathcal{P}^{G_\star}$ for which $k_{\Lambda\Delta}^\Gamma$ of Equation (\ref{eq:str_coef_A_infini}) is a non-zero coefficient. By its definition, $k_{\Lambda\Delta}^\Gamma$ counts the number of pairs of $G$-partial permutations $\big( (g_1;(d_1,\omega_1)),(g_2;(d_2,\omega_2))\big)\in C_{\Lambda;\infty}\times C_{\Delta;\infty}$ such that $$(g_1;(d_1,\omega_1)).(g_2;(d_2,\omega_2))=(g;(d,\omega))$$ where $(g;(d,\omega))$ is a fixed $G$-partial permutation belonging to $C_{\Gamma;\infty}.$ It would be then sufficient to remark that, when multiplying $(g_1;(d_1,\omega_1))$ by $(g_2;(d_2,\omega_2)),$ the permutation $\widetilde{\omega}_{1_{|d_1\cup d_2}} \widetilde{\omega}_{2_{|d_1\cup d_2}}$ acts on at most $|d_1\cup d_2|$ elements. This means that each family of partitions $\Gamma$ that appears in the sum of Equation (\ref{eq:str_coef_A_infini}) must satisfy 
\begin{equation}\label{majoration}
\max(|\Lambda|,|\Delta|)\leq |\Gamma|\leq |\Lambda|+|\Delta|,
\end{equation} 
which ends the proof.
\end{proof}
\begin{remark} More filtrations on $\mathcal{A}^G_\infty$ may exist as suggested by \cite{Ivanov1999}. In this paper, we will only use the above proved one.
\end{remark}
We denote by $\Proj_n$ the natural projection homomorphism between $\mathcal{B}^G_\infty$ and $\mathcal{B}^G_n$ defined on the generating element of $\mathcal{B}^G_\infty$ by
$$\Proj_n(h;(d,\omega))=
\left\{
\begin{array}{ll}
  (h;(d,\omega)) & \qquad \mathrm{if}\quad d\subset [n], \\
  0 & \qquad \mathrm{otherwise.}\quad \\
 \end{array}
 \right.$$
If $\Lambda\in \mathcal{P}^{G_\star},$ then we have: 
$$\Proj_n({\bf C}_{\Lambda;\infty})=
\left\{
\begin{array}{ll}
  {\bf C}_{\Lambda;n} & \qquad \mathrm{if}\quad |\Lambda|\leq n, \\
  0 & \qquad \mathrm{otherwise.}\quad \\
 \end{array}
 \right.$$
By Equation \eqref{eq:str_coef_A_infini}, $k_{\Lambda\Delta}^{\Gamma}$ are also the structure coefficients of the algebra $\mathcal{A}^G_n:$
\begin{equation*}
{\bf C}_{\Lambda;n}{\bf C}_{\Delta;n}=\sum_{\Gamma\in \mathcal{P}^{G_\star}_{\leq n},\atop{ |\Gamma|\leq |\Lambda|+|\Delta|}}k_{\Lambda\Delta}^{\Gamma}{\bf C}_{\Gamma;n},
\end{equation*}
where $\Lambda$ and $\Delta$ are two families of partitions belonging to $\mathcal{P}^{G_\star}_{\leq n}.$ In other words, the structure coefficients $k_{\Lambda\Delta}^{\Gamma}$ of $\mathcal{A}^G_\infty$ do not depend on $n$ and they are the structure coefficients of $\mathcal{A}^G_n$ for any $n.$

\section{Polynomiality of the structure coefficients of $Z(\mathbb{C}[G\wr\mathcal{S}_n])$}\label{sec_6}
In this section we will prove our main result in Theorem \ref{main_theorem}. For this we will need the following lemma.
\begin{lem}\label{lem_main_theo}
If $\Gamma\in \mathcal{P}^{G_\star}_{\leq n}$ then we have ${\bf C}_{\underline{\Gamma}_n}={\bf C}_{\underline{\Gamma^j}_n}$ for any $0\leq j\leq n-|\Gamma|,$ where $\Gamma^j=(\gamma^j(c))_{c\in G_\star}$ is the family of partitions of size $|\Gamma|+j$ indexed by $G_\star$ and defined by $\gamma^j(\lbrace 1_G\rbrace)=\gamma(\lbrace 1_G\rbrace)\cup (1^j)$ and $\gamma^j(c)=\gamma(c)$ if $c\neq \lbrace 1_G\rbrace.$
\end{lem}
\begin{theoreme}\label{main_theorem}
Let $\Lambda, \Delta$ and $\Gamma$ be three proper families of partitions indexed by $G_\star$ and let $n$ be a natural number with $n\geq |\Lambda|, |\Delta|,|\Gamma|.$ The structure coefficient $c_{\Lambda\Delta}^{\Gamma}(n)$ is a polynomial in $n$ with degree $\displaystyle \max_{\overset{0\leq j\leq n-\vert\Gamma\vert}{k_{\Lambda\Delta}^{\Gamma^j}\neq 0}}j$ that can be written:
\begin{equation*}
c_{\Lambda\Delta}^{\Gamma}(n)=\sum_{j=0}^{n-\vert\Gamma\vert} k_{\Lambda\Delta}^{\Gamma^j}\begin{pmatrix}
n-|\Gamma| \\
j
\end{pmatrix},
\end{equation*}
where the coefficients $k_{\Lambda\Delta}^{\Gamma^j}$ are independant integers of $n.$
\end{theoreme}
\begin{proof}
If $\Lambda$ and $\Delta$ are two proper partitions then by Equation \eqref{imagepsi} we have $\psi\big({\bf C}_{\Lambda;n}\big)={\bf C}_{\underline{\Lambda}_n}$ and $\psi\big({\bf C}_{\Delta;n}\big)={\bf C}_{\underline{\Delta}_n}.$
Recall the following equation in $\mathcal{A}^G_n:$
\begin{equation*}
{\bf C}_{\Lambda;n}{\bf C}_{\Delta;n}=\sum_{\Gamma\in \mathcal{P}^{G_\star}_{\leq n},\atop{ |\Gamma|\leq |\Lambda|+|\Delta|}}k_{\Lambda\Delta}^{\Gamma}{\bf C}_{\Gamma;n},
\end{equation*}
Apply $\psi$ to get: 
\begin{equation*}
{\bf C}_{\underline{\Lambda}_n}{\bf C}_{\underline{\Delta}_n}=\sum_{\Gamma\in \mathcal{P}^{G_\star}_{\leq n},\atop{ |\Gamma|\leq |\Lambda|+|\Delta|}}k_{\Lambda\Delta}^{\Gamma}\begin{pmatrix}
n-|\Gamma|+m_1(\gamma(\lbrace 1_G\rbrace)) \\
m_1(\gamma(\lbrace 1_G\rbrace))
\end{pmatrix} {\bf C}_{\underline{\Gamma}_n}.
\end{equation*}
Thus, by Lemma \ref{lem_main_theo}, the right hand side summation of the above equation can be written:
\begin{equation*}
\sum_{\Gamma\in \mathcal{PP}^{G_\star}_{\leq n},\atop{ |\Gamma|\leq |\Lambda|+|\Delta|}} \left[\sum_{j=0}^{n-\vert\Gamma\vert} k_{\Lambda\Delta}^{\Gamma^j}\begin{pmatrix}
n-|\Gamma^j|+m_1(\gamma^j(\lbrace 1_G\rbrace)) \\
m_1(\gamma^j(\lbrace 1_G\rbrace))
\end{pmatrix}\right] {\bf C}_{\underline{\Gamma}_n}.
\end{equation*}
After simplification, we obtain:
\begin{equation*}
\sum_{\Gamma\in \mathcal{PP}^{G_\star}_{\leq n},\atop{ |\Gamma|\leq |\Lambda|+|\Delta|}} \left[\sum_{j=0}^{n-\vert\Gamma\vert} k_{\Lambda\Delta}^{\Gamma^j}\begin{pmatrix}
n-|\Gamma| \\
j
\end{pmatrix}\right] {\bf C}_{\underline{\Gamma}_n},
\end{equation*}
which ends the proof.
\end{proof}

\begin{ex} Let $p>2$ be a prime number and consider $G=\mathbb{Z}_p.$ Recall that the type of $x= (g; p)\in \mathbb{Z}_p \wr \mathcal{S}_n$ is a $p$-vector of partitions $\Lambda=(\lambda_0,\lambda_1,\ldots, \lambda_{p-1})$ where each partition $\lambda_i$ is formed out of cycles $c$ of $p$ whose cycle product equals $i.$ If $0\leq i\leq p-1$ and $\rho$ is a partition then we denote by $\rho^i$ the $p$-vector of partitions $(\lambda_0,\lambda_1,\ldots, \lambda_{p-1})$ where $\lambda_i=\rho$ and $\lambda_s=\emptyset$ if $s\neq i.$ If $\rho$ is a proper partition then it would be clear that $\rho^i$ is a proper family of partitions. We have:
$${\bf C}_{(1)^i;\infty}=\sum_{j\in \mathbb{N}^\star}\big(~ \widehat{i}^j;(\lbrace j\rbrace,\Id_{\lbrace j\rbrace})~\big), $$
where $\widehat{i}^j\in G^\infty_{\lbrace j\rbrace}$ is the vector $x$ with infinite number of coordinates such that $x_j=i$ and $x_v$ is left blank if $v\neq j.$ It would be clear that $(1)^i$ is proper if and only if $0< i\leq p-1.$ Suppose that $j<r$ then we have $$\big(~ \widehat{i}^j;(\lbrace j\rbrace,\Id_{\lbrace j\rbrace})~\big)\big(~ \widehat{t}^r;(\lbrace r\rbrace,\Id_{\lbrace r\rbrace})~\big)=\big(~ \widehat{i}^j+\widehat{t}^r;(\lbrace j,r\rbrace,\Id_{\lbrace j,r\rbrace})~\big).$$
This can be deduced by drawing the multiplication diagram as below:
$$
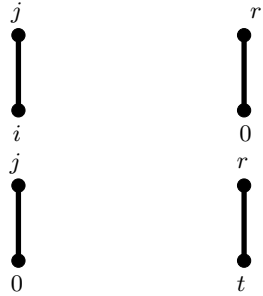
\begin{tikzpicture}[line cap=round,line join=round,>=triangle 45,x=1cm,y=1cm]
\draw [line width=2pt] (-4,3)-- (-4,2);
\draw [line width=2pt] (-1,1)-- (-1,0);
\draw [line width=2pt] (-1,3)-- (-1,2);
\draw [line width=2pt] (-4,1)-- (-4,0);
\begin{scriptsize}
\draw [fill=black] (-4,3) circle (2.5pt);
\draw[color=black] (-4.02,3.3) node {$j$};
\draw [fill=black] (-4,2) circle (2.5pt);
\draw[color=black] (-4.02,1.7) node {$i$};
\draw [fill=black] (-1,1) circle (2.5pt);
\draw[color=black] (-1.02,1.3) node {$r$};
\draw [fill=black] (-1,0) circle (2.5pt);
\draw[color=black] (-1.04,-0.3) node {$t$};
\draw [fill=black] (-1,3) circle (2.5pt);
\draw[color=black] (-0.84,3.3) node {$r$};
\draw [fill=black] (-1,2) circle (2.5pt);
\draw[color=black] (-0.98,1.7) node {$0$};
\draw [fill=black] (-4,1) circle (2.5pt);
\draw[color=black] (-4.04,1.3) node {$j$};
\draw [fill=black] (-4,0) circle (2.5pt);
\draw[color=black] (-4.02,-0.3) node {$0$};
\end{scriptsize}
\end{tikzpicture}
$$
In addition, we have:
$$\big(~ \widehat{i}^j;(\lbrace j\rbrace,\Id_{\lbrace j\rbrace})~\big)\big(~ \widehat{t}^j;(\lbrace j\rbrace,\Id_{\lbrace j\rbrace})~\big)=\big(~ \widehat{i+t}^j;(\lbrace j\rbrace,\Id_{\lbrace j\rbrace})~\big),$$
where the sum $i+t$ is taken modulo $p.$
Thus, if $0< i\leq p-1,$ then
\begin{equation}\label{eq_ex_pol}
{\bf C}_{(1)^i;\infty}{\bf C}_{(1)^t;\infty}=
\left\{
\begin{array}{ll}
  {\bf C}_{(1)^{2i};\infty}+2{\bf C}_{(1^2)^{i};\infty} & \qquad \mathrm{if}\quad t=i, \\
  {\bf C}_{(1)^{i+t};\infty}+2{\bf C}_{(1)^{i}\cup (1)^{t};\infty} & \qquad \mathrm{otherwise,}\quad \\
 \end{array}
 \right.
\end{equation}
where $(1)^{i}\cup (1)^{t}$ is the $p$-vector of partitions $(\lambda_0,\lambda_1,\ldots, \lambda_{p-1})$ with $\lambda_i=(1),$ $\lambda_t=(1)$ and $\lambda_s=\emptyset$ if $s\notin \lbrace i,t\rbrace.$ The coefficient $2$ of ${\bf C}_{(1)^{i}\cup (1)^{t};\infty}$ in the right hand side of Equation (\ref{eq_ex_pol}) is due to the fact that both products $\big(~ \widehat{i}^j;(\lbrace j\rbrace,\Id_{\lbrace j\rbrace})~\big)\big(~ \widehat{t}^r;(\lbrace r\rbrace,\Id_{\lbrace r\rbrace})~\big)$ and $\big(~ \widehat{t}^r;(\lbrace r\rbrace,\Id_{\lbrace r\rbrace})~\big)\big(~ \widehat{i}^j;(\lbrace j\rbrace,\Id_{\lbrace j\rbrace})~\big)$ yields the same $\mathbb{Z}_p$-partial permutation $\big(~ \widehat{i}^j+\widehat{t}^r;(\lbrace j,r\rbrace,\Id_{\lbrace j,r\rbrace})~\big).$ If $n\geq 2,$ by applying $\theta_n$ then $\psi$ on Equation (\ref{eq_ex_pol}) we obtain:

\begin{equation*}
{\bf C}_{\underline{(1)^i}_n}{\bf C}_{\underline{(1)^i}_n}={\bf C}_{\underline{(1)^{2i}}_n}+2{\bf C}_{\underline{(1^2)^{i}}_n} \text{if $i\neq 0$},
\end{equation*}

\begin{equation*}
{\bf C}_{\underline{(1)^i}_n}{\bf C}_{\underline{(1)^t}_n}={\bf C}_{\underline{(1)^{i+t}}_n}+2{\bf C}_{\underline{(1)^{i}\cup (1)^{t}}_n} \text{if $i,t\neq 0,$ $i\neq t$ and $i+t\neq 0$ (mod $p$)},
\end{equation*}
and
\begin{equation*}
{\bf C}_{\underline{(1)^i}_n}{\bf C}_{\underline{(1)^t}_n}={n\choose 1}{\bf C}_{\underline{(1)^{i+t}}_n}+2{\bf C}_{\underline{(1)^{i}\cup (1)^{t}}_n} \text{ if $i,t\neq 0,$ $i\neq t$ and $i+t= 0$ (mod $p$)}.
\end{equation*}
\end{ex}

\section{Irreducible characters of $G\wr \mathcal{S}_n$ and symmetric functions} 

In this section we will recall all the necessary definitions and results from the theory of summetric functions in order to prove that the algebra $\mathcal{A}_\infty^G$ is isomorphic to an algebra of shifted symmetric functions on $|G^\star|$ alphabets.

\subsection{The case $G=\lbrace 1_G\rbrace$} When $G=\lbrace 1_G\rbrace,$ the wreath product $G\wr \mathcal{S}_n$ is isomorphic to the symmetric group $\mathcal{S}_n.$ The irreducible $\mathcal{S}_n$-modules are indexed by partitions of $n.$ If $\lambda\in \mathcal{P}_n,$ we will denote by $V^\lambda$ its associated irreducible $\mathcal{S}_n$-module and by $\chi^\lambda$ its character. It is well known that both the power sum functions $(p_\lambda)_{\lambda}$ and the Schur functions $(s_\lambda)_{\lambda},$ indexed by partitions, are basis families for the algebra $\mathrm{A}$ of symmetric functions. The transition matrix between these two bases is given by the following formula of Frobenius:
\begin{equation} \label{Frob_form}
p_\delta=\sum_{\rho \atop{|\rho|=|\delta|}}\chi^{\rho}_{\delta}s_\rho,
\end{equation}
where $\chi^{\rho}_{\delta}$ denotes the value of the character $\chi^{\rho}$ on any permutation of cycle-type $\delta.$

A shifted symmetric function $f$ in infinitely many variables $(x_1,x_2,\ldots)$ is a family $(f_i)_{i>1}$ that satisfies the following two properties: 
\begin{enumerate}
\item[1.] $f_i$ is a symmetric polynomial in $(x_1 -1,x_2 -2 ,\ldots ,x_i-i).$ 
\item[2.] $f_{i+1}(x_1,x_2,\ldots ,x_i,0) = f_i(x_1,x_2,\ldots ,x_i).$ 
\end{enumerate} 
The set of all shifted symmetric functions is an algebra which we shall denote $\mathrm{A}^{\#}.$ In \cite{okounkov1997shifted}, Okounkov and Olshanski gave a linear isomorphism $\varphi:\mathrm{A}\rightarrow \mathrm{A}^{\#}.$ For any partition $\lambda,$ the images of the power sum function $p_\lambda$ and the Schur function $s_\lambda$ by $\varphi$ are the shifted power symmetric function $p^{\#}_\lambda$ and the shifted Schur function $s^{\#}_\lambda.$ By applying $\varphi$ to the Frobenius relation given in Equation (\ref{Frob_form}), we get:

\begin{equation}\label{shif_Frob_form}
p^{\#}_\delta=\sum_{\rho \atop{|\rho|=|\delta|}}\chi^{\rho}_{\delta}s^{\#}_\rho.
\end{equation}

If $f\in \mathrm{A}^{\#}$ and if $\lambda=(\lambda_1,\lambda_2,\cdots ,\lambda_l)$ is a partition, we denote by $f(\lambda)$ the value $f_l(\lambda_1,\lambda_2,\cdots ,\lambda_l).$ By \cite{okounkov1997shifted}, any shifted symmetric function is determined by its values on partitions. The vanishing characterization of the shifted symmetric functions given in \cite{okounkov1997shifted} states that $s_\rho^{\#}$ is the unique shifted symmetric function of degree at most $|\rho|$ such that 
\begin{equation}\label{characterisation}
s^{\#}_\rho(\lambda)=\left\{
\begin{array}{ll}
      \frac{(|\lambda|\downharpoonright |\rho|)}{\dim \lambda} f^{\lambda/\rho}& \text{ if } \rho\subseteq \lambda \\
      0 & \text{ otherwise } \\
\end{array} 
\right. 
\end{equation}
where $(|\lambda|\downharpoonright |\rho|):=|\lambda|(|\lambda|-1)\cdots (|\lambda|-|\rho|+1)$ is the falling factorial and $f^{\lambda/\rho}$ is the number of skew standard tableaux of shape $\lambda/\rho.$ Using the following branching rule for characters of the symmetric groups
\begin{equation}\label{branching_rule}
\chi^{\lambda}_{\rho\cup (1^{|\lambda|-|\rho|})}=\sum_{\nu; |\nu|=|\rho|}f^{\lambda/\nu}\chi^\nu_\rho,
\end{equation}
one can verify using formulas (\ref{shif_Frob_form}) and (\ref{characterisation}) that
$$p^{\#}_\delta(\lambda)=\left\{
\begin{array}{ll}
      \frac{(|\lambda|\downharpoonright |\delta|)}{\dim \lambda} \chi^{\lambda}_{\underline{\delta}_{|\lambda|}}& \text{ if } |\lambda|\geq |\delta| \\
      0 & \text{ otherwise }. \\
\end{array} 
\right. 
$$
Using this formula, Ivanov and Kerov showed in \cite[Theorem 9.1]{Ivanov1999} that the mapping $F:\mathcal{A}_\infty^{\lbrace 1_G\rbrace}\rightarrow \mathrm{A}^{\#}$ defined on the basis elements of $\mathcal{A}_\infty^{\lbrace 1_G\rbrace}$ by $$F(\mathbf{C}_\delta)=z_\delta^{-1}p^{\#}_\delta$$ is an isomorphism of algebras.

\subsection{The general case}
We refer to \cite[Appendix B]{McDo} for the results of the representation theory of wreath products presented in this section. Let $(P_r(c))_{r\geq 1, c\in G_\star}$ be a family of independent indeterminates over $\mathbb{C}.$ For each $c\in G_\star,$ we may think of $P_r(c)$ as the $r^{th}$ power sum in a sequence of variables $x_c=(x_{ic})_{i\geq 1}.$ Let us denote by $\mathrm{A}^G$ the algebra over $\mathbb{C}$ with algebraic basis the elements $P_r(c)$
$$\mathrm{A}^G:=\mathbb{C}[P_r(c);r\geq 1, c\in G_\star].$$
If $\rho=(\rho_1,\rho_2,\cdots, \rho_l)$ is an arbitrary partition and $c\in G_\star,$ we define $P_\rho(c)$ to be the product of $P_{\rho_i}(c),$
$$P_\rho(c):=P_{\rho_{1}}(c)P_{\rho_2}(c)\cdots P_{\rho_l}(c).$$
The family $(P_\Lambda)_{\Lambda \in \mathcal{P}^{G_\star}},$ where
$$P_\Lambda:=\prod_{c\in G_\star}P_{\Lambda(c)}(c),$$
forms a linear basis for $\mathrm{A}^G.$ That is any element $f\in \mathrm{A}^G$ can be written $f=\sum\limits_{\Lambda \in \mathcal{P}^{G_\star}} f_\Lambda P_\Lambda$ where all but a finite number of the coefficients $f_\Lambda\in \mathbb{C}$ are zero. If we assign degree $r$ to $P_r(c),$ then 
$$\mathrm{A}^G=\bigoplus_{n\geq 0}\mathrm{A}_n^G$$
is a graded $\mathbb{C}$-algebra where $\mathrm{A}_n^G$ is the algebra spanned by all $P_\Lambda$ where $\Lambda\in \mathcal{P}^{G_\star}_n.$ The algebra $\mathrm{A}^G$ can be equipped with a hermitian scalar product defined by 
$$<f,g>=\sum_\Lambda f_\Lambda \bar{g}_\Lambda Z_\Lambda$$
for any two elements $f=\sum\limits_{\Lambda \in \mathcal{P}^{G_\star}} f_\Lambda P_\Lambda$ and $g=\sum\limits_{\Lambda \in \mathcal{P}^{G_\star}} g_\Lambda P_\Lambda$ of $\mathcal{A}^G.$ In particular, we have:
$$<P_\Lambda,P_\Gamma>=\delta_{\Lambda,\Gamma}Z_\Lambda,$$
where $\delta_{\Lambda,\Gamma}$ is the Kronecker symbol. 

For each irreducible character $\gamma\in G^\star$ and each $r\geq 1,$ let 
$$P_r(\gamma):=\sum_{c\in G_\star}\xi_c^{-1}\gamma(c) P_r(c),$$
where $\gamma(c)$ is the value of the character $\gamma$ on an element of the conjugacy class $c.$ By the orthogonality of the characters of $G,$
$$<\gamma,\rho>:=\frac{1}{|G|}\sum_{g\in G}\gamma(g)\rho(g)=\sum_{c\in G_\star}\xi_{c}^{-1}\gamma(c)\rho(c)=\delta_{\gamma,\delta},$$
we can write
$$P_r(c)=\sum_{\gamma\in G^\star}\overline{\gamma(c)}P_r(\gamma).$$
We may think of $P_r(\gamma)$ as the $r^{th}$ power sum in a new sequence of variables $y_\gamma=(y_{i\gamma})_{i\geq 1}$ and denote by $s_\rho(\gamma)$ the schur function $s_\rho$ associated to the partition $\rho$ on the sequence of variables $(y_{i\gamma})_{i\geq 1}.$ Now, for any family of partitions $\Lambda\in \mathcal{P}^{G^\star},$ define $$S_\Lambda:=\prod_{\gamma\in G^\star}s_{\Lambda(\gamma)}(\gamma).$$
The family $(S_\Lambda)_{\Lambda\in \mathcal{P}^{G^\star}}$ is an orthonormal basis of $\mathrm{A}^G,$ see \cite[(7.4)]{McDo}. 

Let $V^\gamma$ be the irreducible $G$-module associated to $\gamma\in G^\star.$ The group $G\wr \mathcal{S}_n$ acts on the $n^{th}$ tensor power $T^n(V^\gamma)=V^\gamma\otimes V^\gamma\cdots \otimes V^\gamma$ as follows:
$$(g;p)\cdot (v_1\otimes v_2\otimes\cdots \otimes v_n):=g_1v_{p^{-1}(1)}\otimes \cdots \otimes g_nv_{p^{-1}(n)},$$
where $(g;p)\in G\wr \mathcal{S}_n$ and $v_1,v_2,\cdots ,v_n\in V^\gamma.$ Let us denote by $\eta_n(\gamma)$ the character of this representation of $G\wr \mathcal{S}_n.$ From \cite[(8.2)]{McDo}, if $x\in G\wr \mathcal{S}_n$ has type $\Lambda\in \mathcal{P}^{G_\star}_n,$ then:
$$\eta_n(\gamma)(x)=\prod_{c\in G_\star}{\gamma(c)}^{l(\Lambda(c))}.$$

For any partition $\mu$ of $m$ and each $\gamma\in G^\star,$ we define
$$X^\mu(\gamma):=\det (\eta_{\mu_i-i+j}({\gamma})).$$ 
We can extend this definition to families of partitions $\Lambda\in \mathcal{P}^{G^\star}_n.$ If $\Lambda\in \mathcal{P}^{G^\star}_n,$ let 
$$X^\Lambda:=\prod_{\gamma\in G^\star}X^{\Lambda(\gamma)}(\gamma).$$ 
The family $(X^\Lambda)_{\Lambda\in \mathcal{P}^{G^\star}_n}$ is a full list of irreducible characters of $G\wr \mathcal{S}_n.$ For any two families of partitions $\Lambda\in \mathcal{P}^{G^\star}_n$ and $\Delta\in \mathcal{P}^{G_\star}_n,$ let us denote by $X^\Lambda_\Delta$ the value of the character $X^\Lambda$ on any of the elements of the conjugacy class $C_\Delta.$ By \cite[page 177]{McDo}, we have the following three important identities

$$X^\Lambda_\Delta=<S_\Lambda,P_\Delta>,$$ 
$$S_\Lambda=\sum_{\Gamma\in \mathcal{P}^{G_\star}_n}Z_{\Gamma}^{-1}X^\Lambda_\Gamma P_\Gamma$$ 
and
$$P_\Delta=\sum_{\Sigma\in \mathcal{P}^{G^\star}_n}\overline{X^\Sigma_\Lambda} S_\Sigma.$$

Let us consider the algebra $\mathrm{A}^{G\#}$ isomorphic to $\mathrm{A}^{G}$ defined using the shifted symmetric functions.  It has a basis formed by the shifted functions $P^{\#}_\Delta$ defined by 
$$P^{\#}_\Delta:=\prod_{c\in G_\star}P^{\#}_{\Delta(c)}(c),$$
for any $\Delta\in \mathcal{P}^{G_\star}.$
For any family of partitions $\Lambda\in \mathcal{P}^{G_\star},$ we set 
$$P^{\#}_\Delta(\Lambda):=\prod_{c\in G_\star}P^{\#}_{\Delta(c)}(c)(\Lambda(c)).$$

\begin{theoreme}\label{th:7.1}
The linear map $F^G:\mathcal{A}_\infty^G\longrightarrow \mathrm{A}^{G\#}$ defined by $$F^G(\mathbf{C}_{\Delta;\infty})=\frac{|G|^{|\Delta|}}{Z_\Delta}P^{\#}_\Delta$$ is an isomorphism of algebras.
\end{theoreme}
\begin{proof}
Let $\Lambda\in \mathcal{P}^{G^\star}_n$ and consider the composition $F_\Lambda^G:=\frac{X^\Lambda}{\dim \Lambda}\circ \psi \circ \Proj_{|\Lambda|}$ of morphisms. Let us see how $F_\Lambda^G$ acts on the basis elements of $\mathcal{A}_\infty^G.$ If $\Delta\in \mathcal{P}^{G_\star}$ with $|\Delta|>|\Lambda|,$ it would be clear then that $F_\Lambda^G({\bf C}_{\Delta;\infty})=0.$ Suppose now that $|\Lambda|\geq |\Delta|,$ we have the following equalities:
\begin{eqnarray}\label{eq_isom}
\left(\frac{X^\Lambda}{\dim \Lambda}\circ \psi \circ \Proj_{|\Lambda|} \right)(\mathbf{C}_{\Delta;\infty})&=& \frac{X^\Lambda}{\dim \Lambda} \left( {|\Lambda|-|\Delta|+m_1(\Delta(\lbrace 1_G\rbrace))\choose m_1(\Delta(\lbrace 1_G\rbrace))} \mathbf{C}_{\underline{\Delta}_{|\Lambda|}} \right) \\
\notag
&=& {|\Lambda|-|\Delta|+m_1(\Delta(\lbrace 1_G\rbrace))\choose m_1(\Delta(\lbrace 1_G\rbrace))}\frac{n!|G|^n}{ Z_{\underline{\Delta}_{|\Lambda|}}\dim \Lambda}X^{\Lambda}_{\underline{\Delta}_{|\Lambda|}} \\
\notag
&=& \frac{(|G|)^{|\Delta|}}{Z_\Delta}\frac{(|\Lambda|\downharpoonright |\Delta|)}{\dim \Lambda} X^{\Lambda}_{\underline{\Delta}_{|\Lambda|}}\\
\notag
&=& \frac{(|G|)^{|\Delta|}}{Z_\Delta}P^{\#}_\Delta(\Lambda)\\
\notag
&=&F^G(\mathbf{C}_{\Delta;\infty})(\Lambda)
\end{eqnarray}
\end{proof}

\bibliographystyle{abbrv}
\bibliography{biblio}
\label{sec:biblio}

\end{document}